\documentclass[a4paper,11pt]{amsart}
\usepackage{a4wide}

\usepackage[utf8]{inputenc}
\usepackage[T1]{fontenc}
\usepackage[american]{babel}
\usepackage{lmodern}
\usepackage{cite}
\usepackage{amsmath}
\usepackage{amssymb}
\usepackage{amsthm}
\usepackage{hyperref}

\def\IR{\mathbb{R}}
\def\IC{\mathbb{C}}

\def\E{\mathcal{E}}
\def\H{\mathcal{H}}
\def\K{\mathcal{K}}

\def\a{\mathfrak{a}}
\def\b{\mathfrak{b}}

\def\lra{\longrightarrow}
\def\downto{\downarrow}

\DeclareMathOperator{\supp}{supp}

\DeclareMathOperator{\sgn}{sgn}
\renewcommand{\Re}{\operatorname{Re}}

\def\loc{\mathrm{loc}}

\providecommand{\abs}[1]{\lvert#1\rvert}
\providecommand{\norm}[1]{\lVert#1\rVert}

\renewcommand{\epsilon}{\varepsilon}
\renewcommand{\phi}{\varphi}

\newtheoremstyle{Beispiel}{}{}{}{}{\bfseries}{:}{ }{}
\theoremstyle{Beispiel}
\newtheorem{example}{Example}[section]
\newtheorem{remark}[example]{Remark}

\newtheoremstyle{Satz}{}{}{\itshape}{}{\bfseries}{:}{ }{}
\theoremstyle{Satz}
\newtheorem{proposition}[example]{Proposition}
\newtheorem{definition}[example]{Definition}
\newtheorem{theorem}[example]{Theorem}
\newtheorem{lemma}[example]{Lemma}
\newtheorem{corollary}[example]{Corollary}

\newcommand{\vol}{\mathrm{vol}}
\newcommand{\End}{\mathrm{End}}

\newcommand{\Hmm}[1]{\leavevmode{\marginpar{\tiny%
$\hbox to 0mm{\hspace*{-0.5mm}$\leftarrow$\hss}%
\vcenter{\vrule depth 0.1mm height 0.1mm width \the\marginparwidth}%
\hbox to 0mm{\hss$\rightarrow$\hspace*{-0.5mm}}$\\\relax\raggedright
#1}}}
\begin{document}
\frenchspacing %kein größerer Abstand nach Satzzeichen als nach Buchstaben
%Titelseite{ex_cone_decreasing_fcts}

\title[Domination of quadratic forms]{Domination of quadratic forms}

\author[Lenz]{Daniel Lenz}
\address{D. Lenz, Mathematisches Institut\\Friedrich-Schiller-Universität Jena\\07737 Jena, Germany}
\email{daniel.lenz@uni-jena.de}

\author[Schmidt]{Marcel Schmidt}
\address{M. Schmidt, Mathematisches Institut\\Friedrich-Schiller-Universität Jena\\07737 Jena, Germany}
\email{schmidt.marcel@uni-jena.de}

\author[Wirth]{Melchior Wirth}
\address{M. Wirth, Mathematisches Institut\\Friedrich-Schiller-Universität Jena\\07737 Jena, Germany}
\email{melchior.wirth@uni-jena.de}

\date{\today}

\begin{abstract}
We study domination of quadratic forms in the abstract setting of
ordered Hilbert spaces. Our main result gives a characterization in
 terms of the associated forms. This generalizes and unifies
various earlier works. Along the way we present several examples.
\end{abstract}

\maketitle

%\blfootnote{2010 \textit{Mathematics Subject Classification.} Primary 47A63, Secondary 35J25, 35R02, 47B65.\\
%\textit{Key words and phrases.} symmetrization, domination, form core, magnetic Schrödinger operator.}

%Inhaltsverzeichnis
\begingroup
\def\addvspace#1{}
\tableofcontents
\endgroup

%%%%%%%%%%%%%%%%%%%%%%%%%%%%%%%%%
%Kapitel 1: Einleitung
\section*{Introduction}
Domination of operators is a way to compare two  operators $A, B$
acting on possibly different Hilbert spaces via an inequality of the
form
$$|A f| \leq B |f|,$$
where the exact meaning of $\abs{\cdot}$ and $\leq$ is  to be
explained later. For now the reader may well think that both Hilbert
spaces are  the same $L^2$-space and $\abs{\cdot}$ just denotes the
usual modulus.

Domination  appears in several disguises and contexts. In connection
with questions of essential self-adjointness, it probably occurred
first in the form of Kato's inequality in the context of comparing
Schr\"odinger operators with and without magnetic field (cf.
\cite{Kat72}). This work has been fairly influential and various
generalizations have been considered subsequently.  Specifically,
Simon showed  in \cite{Sim77} that validity of Kato's inequality for
the generator of a symmetric semigroup in an $L^2$-space is
equivalent to the semigroup being positivity preserving (which can
be understood as being dominated by itself). For two  symmetric
semigroups acting on the same $L^2$-space he showed that domination implies a Kato type inequality for their generators and conjectured these two properties to be equivalent. This conjecture was resolved independently by Simon himself \cite{Sim79b} in the case of two semigroups acting on the same $L^2$-space, and by Hess, Schrader, Uhlenbrock
\cite{HSU} in an even more abstract framework. In that framework the
underlying Hilbert space does not need to be an $L^2$-space and the
two semigroups in question do not need to be defined on the same
Hilbert space. It suffices that the two Hilbert spaces in question
are related via a map -- called absolute pairing or symmetrization --
generalizing the classical modulus function (see \cite{Ber} for a
description of that setting as well).

The preceding works all deal with the generators of the semigroups.
It is also possible (and natural) to study  domination via the
associated forms. A characterization of domination via forms  for
semigroups acting on the same $L^2$-space was first given  by
Ouhabaz  \cite{Ouh96}. For semigroups acting on different
$L^2$-space (one of them even allowed to consist of vector valued
functions)  an analogous  result was then obtained by Manavi, Vogt,
Voigt \cite{MVV}. This result does not require symmetry of the form
but only sectoriality.  While the actual reasoning in the mentioned
works \cite{Ouh96} and \cite{MVV} is somewhat different, they both
rely on a characterization of invariance of a convex set under a
semigroup via the associated form presented in \cite{Ouh96}.

The main  aim  of the present article is to bring together the
general setting of \cite{HSU} with the  point of view of forms given
in \cite{Ouh96,MVV}. Our main  result gives a characterization of
domination of semigroups  in the general framework presented in
\cite{HSU} thereby generalizing corresponding parts of
\cite{Ouh96,MVV}. This rather general setting  opens up the
possibility of new applications of this characterization. To achieve
it we basically follow the strategy of \cite{MVV}. On the technical
level this requires quite some efforts as we can not rely on
pointwise considerations. Indeed,  the main obstacle to overcome is
the lack of points in our setting.

\smallskip

The article is organized as follows: In Section \ref{Background} we
recall the necessary background material and present the setting of
\cite{HSU}. Here,  we take special care to keep the article as
accessible and self-contained as possible by including several
proofs of basic properties. In particular, we give a new and simple
proof of the fact that a self-dual positive cone in a Hilbert space
induces a lattice order if and only if the projection onto the cone
is monotone (Theorem \ref{isotone_latticial}). Section \ref{new}
contains the main new ingredient of our work viz an extension of
Lemma 3.2 of \cite{MVV} to our setting (Proposition
\ref{domination_invariance}). It is there that we deal with the
mentioned lack of points in our setting. Section
\ref{sec_Domination_of_operators} provides the characterization of
dominance (Theorem \ref{thm_char_domination}). Finally, in Section
\ref{Applications} we discuss various classes of examples of
dominated semigroups.  These include perturbation by potentials,
magnetic Schroedinger operators on manifolds and magnetic
Schroedinger operators on graphs. Along the way, we give an
introduction into domination in general and survey various results.

\smallskip

The article has its origin in the master's thesis of one of the
authors (M.~W.).

\bigskip

\textbf{Acknowledgments.} M.~S. and M.~W. gratefully acknowledge
financial support of the DFG via \emph{Graduiertenkolleg: Quanten-
und Gravitationsfelder}. M.~W. gratefully acknowledges financial
support of the \emph{Studienstiftung des deutschen Volkes}. D.~L.
gratefully acknowledges partial support by DFG as well as 
enlightening discussions with Peter Stollmann on Dirichlet forms and
domination of semigroups. M.~W. would like to thank Ognjen Milatovic
for several helpful remarks on a preliminary version of this
article.

\section{Background on positive cones, absolute pairings between Hilbert spaces and domination of
operators} \label{Background} In this section we provide the
necessary background. Most of the material of this section is known.

\subsection{Positive cones and forms satisfying the first Beurling-Deny criterion}\label{Background-positivity} In this section we collect some
basics about order structures in Hilbert spaces induced by a positive
cone. All the material here is certainly well-known -- see
\cite{Nem03} for a (very short) introduction. To make this work
self-contained, we included proofs of the elementary facts.

\smallskip

As it is convenient for order theory, we only deal with vector
spaces over $\IR$ in this section.

\medskip

The basic ingredient for all considerations concerning order in the
present article are positive cones. They are defined next.

\begin{definition}[Positive Cone]
Let $\K$ be a Hilbert space. A closed, non-empty subset $\K_+$ of
$\K$ is called positive cone if
\begin{itemize}
\item[(P1)]$\K_+ +\K_+\subset \K_+$,
\item[(P2)]$ \alpha\K_+\subset \K_+$ for all $ \alpha\geq 0$,
\item[(P3)]$\langle\K_+,\K_+\rangle\geq 0$.
\end{itemize}
The relation  $\leq$ on $\K$  associated to $\K_+$ is given by
$$g_1\leq g_2$$ whenever $g_1,g_2\in \K$ satisfy  $g_2-g_1\in\K_+$.

The positive cone $\K_+$ is said to be self-dual if
\begin{align*}
\K_+=\{g\in \K\mid \langle g,h\rangle\geq 0\text{ for all }h\in\K_+\}.
\end{align*}
The positive cone $\K_+$ is called an isotone projection cone if the
projection $P_{\K_+}$ onto $\K_+$ is monotone increasing with
respect to $\leq$, that is, $g\leq h$ implies $P_{\K_+}(g)\leq
P_{\K_+}(h)$ for $g,h\in\K$.
\end{definition}

\begin{remark}\label{simple-properties}
\begin{itemize}
\item It is not hard to see that $\leq$ is  a partial order
i.e.  reflexive ($f\leq f$), antisymmetric ($f\leq g$ together with
$g\leq f$ implies $f =g$)  and transitive $f\leq g, g\leq h$ implies
$f\leq h$). Moreover, this partial order clearly  makes $\K$ into an
ordered vector space that is  $f\leq g$ implies $f+h\leq g+h$, as
well as $ \alpha f\leq \alpha g$ for all $h\in\K$ and $\alpha \geq
0$ (see Definition \ref{Riesz-space} below as well).

\item  If $\K_+\subset\K$ is a positive cone, the dual cone is defined as
\begin{align*}
\K_+^\circ=\{g\in\K\mid\langle g,h\rangle\geq 0\text{ for all }h\in \K_+\}.
\end{align*}
Hence a cone is self-dual if and only if it coincides with its dual. Sometimes the dual cone is also called polar cone (and self-dual cones are called self-polar), but the usual convention seems to be that the polar cone of $\K_+$ is $-\K_+^\circ$. Note that the dual cone may fail to satisfy (P3) and thus not be a positive cone.
%\item For self-dual cones it is redundant to assume that they are closed.
%Indeed,
%\begin{align*}
%\K_+=\bigcap_{h\in\K_+}\{g\in\K\mid \langle g,h\rangle\geq 0\}
%\end{align*}
%is closed as the intersection of inverse images of a closed set under continuous functions.\\
%In general this is no longer true, as the following example shows: Let $\K=\ell^2$ and $\K_+=\{f\in\ell^2\mid f\geq 0,\,\supp f\text{ finite}\}$. Then $\K_+$ satisfies (P1) -- (P3), but is obviously not closed.
%\item Some authors do not assume property (P3) in the definition of positive cones. However, in the case of self-dual cones we are mainly interested in it is automatically satisfied. Our terminology is in accordance with that of \cite{Ber} and \cite{HSU}, our main sources for Section \ref{sec_Domination_of_operators}.
\item The projection $P_C$ onto a closed, convex subset $C$ of a Hilbert space $H$ maps $x\in H$ to the unique element $P_C(x)\in C$ that satisfies $\norm{P_C(x)-x}=d(x,C)$. It is characterized as the unique $z\in C$ that satisfies $\langle x-z,y-z\rangle\leq 0$ for all $y\in C$.
\end{itemize}
\end{remark}

In most cases we will be interested in the following example.
\begin{example}\label{example_L2}
Let $(X,\mathcal{B},m)$ be a measure space. Then
\begin{align*}
L^2_+(X,m)=\{f\in L^2(X,m)\mid f\geq 0\;m\text{-almost everywhere}\}
\end{align*}
is a self-dual isotone projection cone in $L^2(X,m)$.
\end{example}

\begin{remark}\label{lattice_L2}
It can be shown (cf. \cite{Pen76}, Corollary II.4) that all
self-dual isotone projection cones arise in this way: If $\K_+\subset
\K$ is a self-dual cone that induces a lattice order on $\K$,
then there is a compact space $X$, a regular finite Borel measure $\mu$
on $X$, and a unitary $U\colon \K\lra L^2(X,m)$ such that
$U\K_+=L^2_+(X,m)$. (That self-dual isotone projection cones indeed
satisfy the assumption  of inducing a lattice order is the
content of Theorem \ref{isotone_latticial}.)
\end{remark}

\begin{example}\label{example_matrices}
Let $\K$ be the Hilbert space of all self-adjoint $n\times n$ matrices endowed with the Hilbert-Schmidt inner product. The subset $\K_+$ of all positive semidefinite matrices is a self-dual positive cone in $\K$. It is not an isotone projection cone unless $n=1$.
\end{example}

\begin{example}\label{ex_cone_decreasing_fcts}
Let $\K=L^2([0,\infty),r\,dr)$. The subset
\begin{align*}
\K_+=\{f\in L^2([0,\infty),r\,dr)\mid f\geq 0,\,f'\leq 0\},
\end{align*}
where $f'$ is the distributional derivative, is a positive cone in $\K$. It is not self-dual.
\end{example}

An important tool when dealing with positive cones is Moreau's
Theorem (cf. \cite{M62}). It shows that in Hilbert spaces with a
positive cone there is a decomposition
$$ g =P_{\K_+}(g) -P_{\K_+^\circ}(-g).$$
This decomposition can be seen as providing an abstract analog of
the decomposition of a function into positive and negative part.
Indeed, based on the theorem we will subsequently discuss various
instances of this analogy. We include a proof of Moreau's theorem
for completeness sake.

\begin{theorem}[Moreau]\label{Moreau}
Let $\K$ be a Hilbert space and $\K_+\subset\K$ a positive cone. For $g,h_1,h_2\in\K$ the following statements are equivalent:
\begin{itemize}
\item[(i)]$g=h_1-h_2$, $h_1\in\K_+$, $h_2\in \K_+^\circ$, $\langle h_1,h_2\rangle=0$
\item[(ii)]$h_1=P_{\K_+}(g)$, $h_2=P_{\K_+^\circ}(-g)$.
\end{itemize}
\end{theorem}
\begin{proof}
(i)$\implies$(ii): Let $h_1$, $h_2$ be as in (i) and $h\in \K_+$. Then
\begin{align*}
\langle g-h_1,h-h_1\rangle=\langle -h_2,h-h_1\rangle=-\langle h_2,h\rangle\leq 0.
\end{align*}
Thus $h_1=P_{\K_+}(g)$. The proof for $h_2=P_{\K_+^\circ}(-g)$ is analogous.

(ii)$\implies$(i): Let $h_1$, $h_2$ be as in (ii). By definition, $h_1\in\K_+$ and $h_2\in \K_+^\circ$. Moreover,
\begin{align*}
\langle g-h_1,h-h_1\rangle\leq 0
\end{align*}
for all $h\in \K_+$. Inserting $h=0$ and $h=2h_1$, we get $\langle g-h_1,h_1\rangle=0$.

Hence
\begin{align*}
0\geq \langle g-h_1,h-h_1\rangle=\langle g-h_1,h\rangle
\end{align*}
for all $h\in \K_+$, which implies $h_1-g\geq 0$.

Finally, we infer from
\begin{align*}
\langle -g-(h_1-g),h-(h_1-g)\rangle=\langle -h_1,h-(g-h_1)\rangle=-\langle h_1,h\rangle\leq 0
\end{align*}
for all $h\in\K_+^\circ$ that $h_1-g=P_{\K_+^\circ}(-g)=h_2$.
\end{proof}

Here is  a first consequence of Moreau's theorem giving basic
properties of the decomposition $g =P_{\K_+}(g)
-P_{\K_+^\circ}(-g)$.

\begin{lemma}\label{norm_Riesz_monotone}
Let $\K$ be a Hilbert space and $\K_+\subset \K$ a positive cone.
Then $ \norm\cdot\colon \K_+\lra[0,\infty)$ is monotone increasing
and $\norm{g}=\norm{P_{\K_+}(g)+P_{\K_+^\circ}(-g)}$ for all
$g\in\K$.
\end{lemma}
\begin{proof}
Let $g,h\in \K_+$ such that $g\leq h$. Then we have
\begin{align*}
0\leq \langle h-g,g\rangle=\langle h,g\rangle-\norm g^2\leq \norm
h\norm g-\norm g^2,
\end{align*}
hence $\norm g\leq\norm h$. Moreover,
\begin{align*}
\norm{g}^2&=\norm{
P_{K_+}(g)-P_{\K_+^\circ}(g)}^2=\norm{P_{\K_+}(g)+P_{\K_+^\circ}(-g)}^2
\end{align*}
since $\langle P_{\K_+}(g),P_{\K_+^\circ}(-g)\rangle=0$ by Moreau's
Theorem \ref{Moreau}.
\end{proof}

We now turn to studying the order structure coming from a positive
cone discussed in  Remark \ref{simple-properties}. In particular, we
will discuss the connection between Riesz spaces (to be defined
next) and positive cones.

\begin{definition}[Riesz space]\label{Riesz-space}
A vector space $E$ with a partial order $\leq$ is called ordered
vector space if for all $f,g,h\in E$, $ \alpha\geq 0$ the following
properties hold:
\begin{itemize}
\item[(R1)]$f\leq g$ implies $f+h\leq g+h$,
\item[(R2)]$f\leq g$ implies $ \alpha f\leq  \alpha g$.
\end{itemize}
If additionally
\begin{itemize}
\item[(R3)]$\{f,g\}$ has a least upper bound $f\vee g$,
\end{itemize}
 for all $f,g\in E$, then $E$ is called a Riesz space. A subspace
$F$ of a Riesz space $E$ is called sublattice if $f,g\in F$ implies
$f\vee g\in F$.
\end{definition}
\begin{remark}
The partial order induced by the cone from Example \ref{example_L2} makes $L^2(X,m)$ into a Riesz space with minima and maxima given by the respective pointwise almost everywhere operations.

The space of self-adjoint $n\times n$ matrices with the partial order induced by the cone from Example \ref{example_matrices} is not a Riesz space unless $n=1$ (compare also Theorem \ref{isotone_latticial}).
\end{remark}

If $E$ is an ordered vector space, then  a least  upper bound can
easily be seen to be  necessarily unique (if it exists). Similarly,
a greatest lower bound will be unique if it exists. It will be
denoted by $f\wedge g$ (if it exists).  Clearly, for any $f,g$ in a
Riesz space both least upper bounds and greatest lower bound exist
and are unique.

\begin{definition}[Positive and negative part]
Let $E$ be an ordered vector space. For $f\in E$ the positive and
negative part are defined as $f_{\pm}=(\pm f)\vee 0$ if they exist.
In this case the absolute value is defined as $\abs{f}=f_+ +f_-$.
\end{definition}

\begin{remark}
Let $E$ be an ordered vector space. If $f\vee 0$ exists for all $f\in E$, then its is not hard to see that $E$ is a Riesz space and the lattice operations are given by
\begin{align*}
f\wedge g&=\frac 1 2(f+g-\abs{f-g}),\\
f\vee g&=\frac 1 2(f+g+\abs{f-g})
\end{align*}
for all $f,g\in E$.
\end{remark}

It is very natural to think about  $f_{\pm}$ as analogs of positive
and negative part of a function. So, in Hilbert spaces with order
coming from a positive cone there  are now  both $f_+$ (if it
exists) and $P_{\K_+}(f)$ obvious  candidates for the positive part
of $f$. Luckily, these two agree (under an additional condition).

\begin{lemma}\label{positive_part_projection}
Let $\K$ be a Hilbert space, $\K_+\subset \K$ a self-dual isotone
projection cone, and $g\in\K$. Then $P_{\K_+}(g)$ is the least upper
bound of $\{0,g\}$.
\end{lemma}
\begin{proof}
Let $g\in\K$. By Moreau's Theorem \ref{Moreau} we have
\begin{align*}
g=P_{\K_+}(g)-P_{\K_+}(-g)\leq P_{\K_+}(g).
\end{align*}
Hence, $P_{K_+}(g)$ is an upper bound for $\{0,g\}$. Now let $h$ be an upper bound for $\{0,g\}$. By isotonicity we have $P_{\K_+}(g)\leq P_{\K_+}(h)=h$. Thus, $P_{\K_+}(g)$ is the least upper bound of $\{0,g\}$.
\end{proof}

\begin{lemma}\label{projection_Riesz}
Let $\K$ be a Hilbert space and $\K_+\subset \K$ a self-dual
positive cone such that $(\K,\leq)$ is a Riesz space. If $g\in\K$,
then $g_+=P_{\K_+}(g)$ and $g_-=P_{\K_+}(-g)$.
\end{lemma}
\begin{proof}
First note that $g_+,g_-\in \K_+$ and $g=g_+ - g_-$ by \cite{Sch71},
V.1.1. By Moreau's theorem \ref{Moreau} it suffices to show that
$\langle g_+,g_-\rangle=0$.

As $g=P_{\K_+}(g)-P_{\K_+}(-g)\leq P_{\K_+}(g)$ and $P_{\K_+}(g)\geq
0$ by definition, we have $g_+\leq P_{\K_+}(g)$. Analogously,
$g_-\leq P_{\K_+}(-g)$. Thus
\begin{flalign*}
&&0\leq \langle g_+,g_-\rangle\leq \langle P_{\K_+}(g),P_{\K_+}(-g)\rangle=0.&&\qedhere
\end{flalign*}
\end{proof}

After these preparations we can now unravel  the connection between
Riesz spaces and positive cones in Hilbert spaces.  The implication
that every isotone projection cone induces a lattice order is due to
Isac and Németh (see \cite{IN}, Proposition 3), the converse
implication due to Németh (see \cite{Nem03}, Theorem 3). Both
results do not assume the cone to be self-dual. As the proof is
considerably simpler in the self-dual case, which is our main
interest,  we include it here for that case only.

\begin{theorem}\label{isotone_latticial}
Let $\K$ be a Hilbert space and $\K_+\subset\K$ a self-dual positive
cone. Then $(\K,\leq)$ is a Riesz space if and only if $\K_+$ is an
isotone projection cone.
\end{theorem}
\begin{proof}
First assume that $(\K,\leq)$ is a Riesz space. By Lemma
\ref{projection_Riesz}, $P_{\K_+}(g)=g_+$ for all $g\in \K$. If
$g\leq h$, then $g_+\leq h_+$ follows immediately from the
definition.

Now assume that $\K_+$ is an isotone projection cone. It is easy to
see that it suffices to show that $\{0,g\}$ has a least upper bound
for all $g\in \K$, and this was proven in Lemma
\ref{positive_part_projection}.
\end{proof}

Having studied some of the basic order properties of Hilbert spaces,
we will now turn to forms on ordered Hilbert spaces that are
compatible with this order structure. By a (quadratic form) we
always mean a densely defined, lower bounded quadratic form.

\begin{definition}[First Beurling-Deny criterion]
Let $\K$ be a Hilbert space and $\K_+\subset\K$ a positive cone. A form $\b$ in $\K$ is said to satisfy the first Beurling-Deny criterion if $P_{\K_+}D(\b)\subset D(\b)$ and
\begin{align*}
\b(P_{\K_+}(g),P_{\K_+^\circ}(-g))\leq 0
\end{align*}
for all $g\in D(\b)$.
\end{definition}

As a corollary of the following result by Ouhabaz (see \cite{Ouh96}, Thm. 2.1 and
Proposition 2.3, and \cite{Ouh99}, Theorem 3), a form satisfies the first Beurling-Deny criterion if
and only if the associated semigroup preserves the positive cone
$\K_+$. Note that in the following proposition, unlike in the rest of this section, the Hilbert space $\H$ may be real or complex.

\begin{proposition}[Ouhabaz]\label{invariance_ouhabaz}
Let $\H$ be a Hilbert space, $C$ a closed, convex subset of $\H$,
$P$ the projection onto $C$, $(Q_t)$ a semigroup on $\H$ with
generator $T$ and $q$ the associated form with lower bound
$-\lambda$. Then the following are equivalent:
\begin{itemize}
\item[(i)]$Q_t C\subset C$ for all $t\geq 0$
\item[(ii)]$\alpha(T+\alpha)^{-1}C\subset C$ for all $\alpha>\lambda$
\item[(iii)]$P(D(q))\subset D(q)$ and $\Re q(Pu,u-P u)\geq 0$ for all $u\in D(q)$
%\item[(iv)]$P(D(q))\subset D(q)$ and $\Re q(u,u-Pu)\geq -\lambda \norm{u-Pu}^2$ for all $u\in D(q)$
\end{itemize}
\end{proposition}

\begin{corollary}[Characterization of positivity preserving semigroups] \label{char-pos-form-via-cone}
Let $\K$ be a Hilbert space and $\K_+\subset\K$ a positive cone. Let
$\b$ be a closed form in $\K$ and $B$ the associated self-adjoint operator. Then $\b$ satisfies the first Beurling-Deny criterion if and only if $e^{-t B}$ leaves $\K_+$ invariant for all
$t\geq 0$.
\end{corollary}

 The next lemma gives another characterization of forms satisfying the first Beurling-Deny criterion.
We omit the proof since it is well-known for $L^2$-spaces and easily
carries over to our more abstract setting.

\begin{lemma}\label{positive_form_Riesz}
Let $\K$ be a Hilbert space, $\K_+\subset \K$ a self-dual isotone
projection cone, and $\b$ a form in $\K$. Then $\b$ satisfies the first Beurling-Deny criterion if and only if $\abs{D(\b)}\subset D(\b)$ and
$\b(\abs{g})\leq \b(g)$ for all $g\in D(\b)$.
\end{lemma}

Some basic properties of forms satisfying the first Beurling-Deny
criterion and their domains are collected in the following lemma.

\begin{lemma}\label{pos_form_lattice}
Let $\K$ be a Hilbert space, $\K_+\subset\K$ a self-dual isotone
projection cone, and $\b$ a form in $\K$ satisfying the first Beurling-Deny criterion with lower
bound $- \lambda\in\IR$.

\begin{itemize}
\item[(a)] The form domain $D(\b)$ is a sublattice of $\K$.
\item[(b)] For any $\alpha\geq \lambda$, the form $\b_ \alpha:=\b+
\alpha\langle\cdot,\cdot\rangle$ satisfies
\begin{align*}
\b_ \alpha(g\wedge h),\b_ \alpha(g\vee h)\leq \b_ \alpha(g)+\b_ \alpha(h)
\end{align*}
for all $g,h\in D(\b)$.
\end{itemize}
\end{lemma}
\begin{proof}
Let $g,h\in D(\b)$. By Lemma \ref{positive_form_Riesz} we have
\begin{align*}
g\wedge h=\frac 1 2(g+h-\abs{g-h})\in D(\b)
\end{align*}
and analogously for $g\vee h$. Hence, $D(\b)$ is a sublattice of $\K$.

Since $\langle P_{\K_+}(g),P_{\K_+}(-g)\rangle=0$ by Moreau's Theorem \ref{Moreau}, the form $b_ \alpha$ satisfies the first Beurling-Deny criterion. Moreover, $\b_ \alpha(u+v)\geq 0$ implies $-2\b_ \alpha(u,v)\leq \b_ \alpha(u)+\b_ \alpha(v)$ for all $u,v\in D(\b)$. With the aid of this inequality, the positivity of $\b_ \alpha$ and the parallelogram identity we obtain
\begin{align*}
\b_ \alpha(g\wedge h)&=\frac 1 4\b_ \alpha(g+h-\abs{g-h})\\
&=\frac 1 4(\b_\alpha(g+h)+\b_\alpha\abs{g-h})-2\b_\alpha(g+h,\abs{g-h}))\\
&\leq \frac 1 2(\b_\alpha(g+h)+\b_\alpha(\abs{g-h})\\
&\leq\frac 1 2(\b_\alpha(g+h)+\b_\alpha(g-h))\\
&=\b_\alpha(g)+\b_\alpha(h).
\end{align*}
The result for $\b_\alpha(g\vee h)$ follows similarly.
\end{proof}

\begin{remark}[Real versus complex Hilbert spaces]
So far we developed a theory developed for real Hilbert spaces only as
that  completely serves our purposes. To incorporate complex Hilbert spaces, one can proceed as follows:

Every complex Hilbert space $\K$ becomes a real Hilbert space
$\K_{\mathrm{r}}$ when equipped with the inner product
$\langle\cdot,\cdot\rangle_{\mathrm{r}}=\Re\langle\cdot,\cdot\rangle$
and a positive cone $\K_+$ in $\K$ is also a positive cone in
$\K_{\mathrm{r}}$. However, self-duality of $\K_+$ is not preserved,
but $\K_{\mathrm{r}}$ decomposes as
\begin{align*}
\K_{\mathrm{r}}=\K^J\oplus i\K^J
\end{align*}
with $\K^J=\K_+-\K_+$, and $\K_+$ is a self-dual cone in $\K^J$. Therefore, every element $g\in \K_+$ has a unique decomposition as
\begin{align*}
g=g_1-g_2+i(g_3-g_4)
\end{align*}
with $g_1,\dots,g_4\in\K_+$ and $\langle g_i,g_j\rangle=0$ for
$i\neq j$. This decomposition yields an anti-unitary involution $J$
via
\begin{align*}
J\colon \K^J\oplus i\K^J\lra\K^j\oplus  i\K^J,\,g+ih\mapsto g-ih.
\end{align*}
Forms satisfying the first Beurling-Deny criterion are real in the sense that $JD(\b)=D(\b)$ and
$\b(Jg)=\b(g)$. Thus, there is no loss of generality when dealing
exclusively with real Hilbert spaces as we can always restrict to
$\K^J$ in the complex case.
\end{remark}

\subsection{Absolute pairings and domination of
operators}\label{Symmetrization}
 In this section we introduce the
concept of domination of operators. As mentioned already this
concept allows one to compare two operators. These operators may act
on different Hilbert spaces provided that there is a modulus type
map between these Hilbert spaces. Such a map is called
absolute pairing.

\smallskip

Throughout this section $\K$  denotes a real Hilbert space and $\H$
a Hilbert space, either real or complex.

\medskip

\begin{definition}[Absolute mapping, absolute pairing]
Let $\K_+\subset \K$ be a positive cone. A map $S\colon \H\lra\K_+$ is called absolute mapping if
\begin{itemize}
\item[(S1)]$\abs{\langle f_1,f_2\rangle}\leq\langle S(f_1),S(f_2)\rangle$ for all $f_1,f_2\in\H$ with equality if $f_1=f_2$.
\end{itemize}
An absolute mapping $S$ is called absolute pairing (or symmetrization) if
\begin{itemize}
\item[(S2)]
For all $g\in\K_+$ and $f_1\in\H$ there exists $f_2\in\H$ such that
$g=S(f_2)$ and
\begin{align*}
\langle f_1,f_2\rangle=\langle S(f_1),S(f_2)\rangle
\end{align*}
In this case $f_1$ and $f_2$ are called paired.
\end{itemize}
\end{definition}

\begin{remark}
 By its very definition every absolute pairing is surjective.
\end{remark}

If one  thinks of an absolute mapping $S$  as a form of modulus,
then one may think of $f_2$ appearing in (S2) of the preceding
definition as a form of $g\sgn (f_1)$ (see the examples
below for further justification of this point of view). In this
sense an absolute pairing is a modulus type map together with the
possibility of forming a signum.

\smallskip

The following lemma shows that we have already encountered a natural
example of an absolute mapping in the last section.
\begin{lemma}\label{absolute_mapping_Riesz}
If $\K_+\subset\K$ is a self-dual positive cone, then
\begin{align*}
S\colon\K\lra\K_+,\,g\mapsto P_{\K_+}(g)+P_{\K_+}(-g)
\end{align*}
is an absolute mapping. If $\K_+$ is an isotone projection cone, then $S=\abs{\cdot}$ is an absolute pairing.
\end{lemma}
\begin{proof}
Let $g,h\in\K$. Then we have
\begin{align*}
\abs{\langle g,h\rangle}&=\abs{\langle P_{\K_+}(g)-P_{\K_+}(-g),P_{\K_+}(h)-P_{\K_+}(-h)\rangle}\\
%&=\abs{\langle P_{\K_+}(g),P_{\K_+}(h)\rangle-\langle P_{\K_+}(-g),P_{\K_+}(h)\rangle-\langle P_{\K_+}(g),P_{\K_+}(-h)\rangle+\langle P_{\K_+}(-g),P_{\K_+}(-h)\rangle}\\
&\leq \langle P_{\K_+}(g),P_{\K_+}(h)\rangle+\langle P_{\K_+}(-g),P_{\K_+}(h)\rangle+\langle P_{\K_+}(g),P_{\K_+}(-h)\rangle\\
&\quad+\langle P_{\K_+}(-g),P_{\K_+}(-h)\rangle\\
&=\langle P_{\K_+}(g)+P_{\K_+}(-g),P_{\K_+}(h)+P_{\K_+}(-h)\rangle\\
&=\langle S(g),S(h)\rangle.
\end{align*}
Equality in the case $g=h$ was already shown in Lemma \ref{norm_Riesz_monotone}.

If $\K_+$ is an isotone projection cone, then $\K$ is order isomorphic to $L^2(X,m)$ for some measure space $m$ (see Remark \ref{lattice_L2}) and the claim follows easily (see also Example \ref{ex_sym_L2}).
\end{proof}

In a spatial setting the modulus together with a signum function
provide absolute pairings as discussed  in the next examples.

\begin{example}
The norm $\norm{\cdot}\colon\H\lra [0,\infty)$ is an absolute pairing.
For $\lambda>0$ and $f_1\in\H$ an element $f_2\in\H$ such that $f_1$
and $f_2$ are paired and $\norm{f_2}=\lambda$ is given by
$f_2=\lambda\frac{f_1}{\norm{f_1}}$ if $f_1\neq 0$, and by
$f_2=\lambda\xi$ for any $\xi\in\H$ with $\norm{\xi}=1$ if $f_1=0$.
\end{example}

\begin{example}[Direct integrals]\label{ex_sym_L2}
Let $((H_x)_{x\in X},\mathfrak{M})$ be a measurable field of Hilbert
spaces over $(X,\mathcal{B},\mu)$ in the sense of \cite{Tak02},
Definition IV.8.9, and $\H=\int_X^\oplus H_x\,d\mu(x)$. The norm on
$H_x$ is denoted by $\abs{\cdot}_x$, $x\in X$.
 Then, the
map $S\colon \H\lra L^2_+(X,\mu)$ given by
$S(\xi)(x)=\abs{\xi(x)}_x$ is an absolute pairing. Indeed, for
$\xi\in\H$ and $f\in L^2_+(X,\mu)$ let
\begin{align*}
\eta(x)=\begin{cases}\frac{ f(x)}{\abs{\xi(x)}_x}  \xi(x) &\colon \xi(x)\neq 0,\\
f(x)\zeta(x)&\colon \xi(x)=0,\end{cases}
\end{align*}
where $\eta\in\mathfrak{M}$ with $S(\zeta)=1$ (the existence of
such an element is proven in \cite{Tak02}, Lemma IV.8.12). Then
$\eta$ and $\xi$ are paired with $S(\eta)=f$. The other
properties of an absolute pairing are easy to check.

A special case of this construction is given by a constant field of
Hilbert spaces, where all  $H_x$, $x\in X$,  are equal to one fixed
separable Hilbert space $H$. In this case, $\H$ is also denoted by
$L^2 (X,\mu;H)$ and given by the vector space of all measurable
maps (with respect to the corresponding Borel-$\sigma$-algebras)
$\xi : X\longrightarrow H$ with $\int \abs{\xi(x)}_x^2 d\mu (x)
<\infty$, where two such maps are identified if they agree
$\mu$-almost everwhere.
\end{example}

\begin{example}\label{herm_vector_bundle}
Let $X$ be a topological space, $m$ a Borel measure on $X$ and $E$ a
Hermitian vector bundle over $X$ with canonical projection $\pi :
E\longrightarrow X$. Loosely speaking this means that $E$ is  a
vector bundle with an inner product on the fibers that varies
continuously with the base point. More precisely, each fiber
$\pi^{-1} (x)$, $x\in X$, carries the structure of a finite vector
space with an inner product $\langle \cdot,\cdot\rangle_x$ and the
bundle is locally trivial with a fixed  finite dimensional Hilbert
space $(H,\langle\cdot,\cdot\rangle)$ as model, that is, to each
point $p\in X$ there exists a neighborhood $U$ and a homeomorphism
$\varphi : U\times H\longrightarrow \pi^{-1} (U)$, called local
trivialization, such that $\varphi_x :=\varphi (x,\cdot)$ is an an
isometric isomorphism between the inner product space $H$ and
$\pi^{-1}(x)$ for each $x\in U$.  A function $\eta :
X\longrightarrow E$ with $\pi \circ \xi = \mathrm{id}_X$ is called a
section. By $L^2 (X,\mu;E)$ we denote the set of measurable (with
respect to the corresponding Borel-$\sigma$-algebras) sections $\xi$
with $\int_X \abs{\xi(x)}_x^2   dm (x) < \infty$, where
$\abs{\cdot}_x$ is the norm induced from $\langle
\cdot,\cdot\rangle_x$ and sections are identified which agree
$\mu$-almost everywhere. This space is a Hilbert space. Assume  that
the bundle admits one nowhere vanishing measurable section $\xi$.
Then
\begin{align*}
S\colon L^2(X,m;E)\lra L^2_+(X,m),\,(Sf)(x)=\abs{f(x)}_x
\end{align*}
is an absolute pairing by the same arguments as in the previous
example. The assumption that the bundle admits a nowhere vanishing
section can easily be seen to be met if the underlying space is
$\sigma$-compact or a separable metric space. More generally, it
suffices that the underlying space satisfies the Lindelöf property
that every cover has a countable subcover.
\end{example}

\begin{example}
For a Borel subset $A$ of $\IR^n$ denote by $A^\ast$ the ball in $\IR^n$ with the same volume as $A$. For $f\in L^2(\IR^n)$ and $r\geq 0$ define
\begin{align*}
f^\ast(r)=\int_0^\infty 1_{(f^{-1}(t,\infty))^\ast}(x)\,dt,
\end{align*}
where $x\in \IR^n$ with $\abs{x}=r$ (obviously, $f^\ast(r)$ does not depend on the choice of $x$). The function $f^\ast$ is called the decreasing rearrangement of $f$.

Now let $\H=L^2(\IR^n)$ and $\K=L^2([0,\infty),r\,dr)$ with the positive cone $\K_+=\{f\in \K\mid f\geq 0,\,f'\leq 0\}$ from Example \ref{ex_cone_decreasing_fcts}. Then
\begin{align*}
S\colon \H\lra \K_+,\,f\mapsto f^\ast
\end{align*}
is an absolute pairing (see \cite{Ber}, Theorem 25 of Appendix A).
\end{example}

In the following lemma we collect some basic properties of the
modulus that carry over to abstract absolute pairings. See \cite{HSU},
Proposition 2.6, for a proof.

\begin{lemma}\label{symmetrization_triangle}
Let $\K_+\subset\K$ be a positive cone and $S\colon\H\lra\K_+$ an absolute pairing.
\begin{itemize}
\item[(a)]The triangle inequality
\begin{align*}
\langle S(f_1+f_2),g\rangle\leq \langle S(f_1)+S(f_2),g\rangle
\end{align*}
holds for all $f_1,f_2\in\H$ and $g\in\K_+$.
\item[(b)]The map $S$ is positive homogeneous, i.e.
\begin{align*}
S(\alpha f)=\abs{\alpha} S(f)
\end{align*}
holds for all $f\in\H$ and $\alpha\in\IC$.
\item[(c)]The map $S$ is positive definite: For all $f\in\H$, $S(f)=0$ if and only if $f=0$.
\item[(d)]The map $S$ is Lipschitz continuous.
\end{itemize}
\end{lemma}

To a certain extent absolute pairings are compatible with taking
differences (as it the case for the usual modulus).

\begin{lemma}\label{abs_difference}
Let $\K_+\subset\K$ be a positive cone and $S\colon \H\lra\K_+$ an absolute pairing. If $f_1,f_2\in\H$ are paired and satisfy $S(f_2)\leq S(f_1)$, then $S(f_1-f_2)=S(f_1)-S(f_2)$ and $f_1-f_2$ and
$f_2$ are paired.
\end{lemma}
\begin{proof}
By the Lemma \ref{symmetrization_triangle}, the triangle inequality $\langle S(f_1-f_2),h\rangle\geq \langle S(f_1)-S(f_2),h\rangle$ holds for all $h\in\K_+$. Moreover we have
\begin{align*}
\norm{S(f_1-f_2)}^2&=\norm{f_1-f_2}^2\\
&=\norm{f_1}^2+\norm{f_2}^2-2\langle f_1,f_2\rangle\\
&=\norm{S(f_1)}^2+\norm{S(f_2)}^2-2\langle S(f_1),S(f_2)\rangle\\
&=\norm{S(f_1)-S(f_2)}^2.
\end{align*}
Let $g,g'\in\K_+$ such that $\langle g,h\rangle\leq \langle g',h\rangle$ for all $h\in \K_+$ and $\norm g=\norm{g'}$. Then
we have
\begin{align*}
\norm{g-g'}^2 =\norm{g}^2+\norm{g'}^2-2\langle g,g'\rangle =2\norm{g}^2-2\langle g,g+g'-g\rangle=-2\langle g,g'-g\rangle \leq 0.
\end{align*}
Hence $g=g'$. Applying this
result to $g=S(f_1)-S(f_2)$ and $g'=S(f_1-f_2)$, we obtain the desired
equality for $S(f_2-f_1)$. Moreover,
\begin{align*}
\langle f_2-f_1,f_2\rangle=\langle S(f_2)-S(f_1),S(f_2)\rangle=\langle S(f_1-f_2),S(f_2)\rangle,
\end{align*}
hence $f_1-f_2$ and $f_2$ are paired.
\end{proof}

The next lemma will serve as a characterization of the central
concept of this section, namely domination of operators. A proof is
given in \cite{Ber}, Proposition 11 of Appendix A.
\begin{lemma}\label{abs_char_domination}
Let $\K_+\subset \K$ be a positive cone and $S\colon\H\lra\K_+$ an absolute pairing. For bounded operators $P$ (resp. $Q$) on $\H$ (resp. $\K$), the following are equivalent:
\begin{itemize}
\item[(i)]$\langle S(P f_1),g\rangle\leq \langle Q S (f_1),g\rangle$ for all $f_1\in \H,\,g\in\K_+$
\item[(ii)]$\Re\langle  Pf_1,f_2\rangle\leq \langle Q S(f_1),S(f_2)\rangle$ for all $f_1,f_2\in\H$
\item[(iii)]$\abs{\langle P f_1,f_2\rangle}\leq \langle Q S(f_1),S(f_2)\rangle$ for all $f_1,f_2\in\H$
\end{itemize}
Furthermore, if $\K_+$ is self-dual, these assertions are equivalent to
\begin{itemize}
\item[(iv)]$S(P f_1)\leq Q S(f_1)$ for all $f_1\in \H$.
\end{itemize}
\end{lemma}

\begin{definition}[Domination of operators]\label{def_domination}
If $P$ and $Q$ satisfy one of the equivalent assertions of Lemma
\ref{abs_char_domination}, then $P$ is said to be dominated by $Q$.
A family of bounded operators $(P_\alpha)_{\alpha\in J}$
is said to be dominated by the family of bounded operators
$(Q_\alpha)_{\alpha\in J}$ if $P_\alpha$ is dominated by $Q_\alpha$
for each  $\alpha\in J$.
\end{definition}

By its very definition the  domination of $P$ by $Q$  can be read as
the operator $Q$ entailing  properties of the operator $P$ (and this
will be our point of view in our main result below). Note, however,
that the definition also implies some structural positivity property
of $Q$ as $Q (\K_+)$ must be  a subset of  $\K_+$. We will meet
this positivity property in various places below.

\begin{example}\label{ex_positive_semigroup_dominating}
Let $\K$ be a Hilbert space, $\K_+\subset \K$ a self-dual isotone projection cone, and $P\colon \K\lra \K$ a bounded linear operator that leaves $\K_+$ invariant. Then
$P$ is dominated by itself: For all $g\in \K$, we have
\begin{align*}
\abs{Pg}=\abs{Pg_+ - Pg_-}\leq \abs{Pg_+}+\abs{Pg_-}=P g_+ +P g_-=P\abs{g}.
\end{align*}
Indeed, also the converse is true: If $P$ is dominated by itself, then $Pg=P\abs{g}\geq \abs{Pg}\geq 0$ for all $g\in \K_+$, hence $P \K_+\subset \K_+$.
\end{example}

We will give more interesting examples (in particular such that have
different operators $P$ and $Q$) once we have a characterization of
domination of semigroups in terms of the associated forms at hand.
But before we turn to this characterization, we present some basic
algebraic properties of domination. A proof is given in \cite{Ber},
Appendix A, Lemma 14.

\begin{lemma}\label{domination_sums}
Let $\K_+\subset \K$ be a positive cone, $S\colon\H\lra\K_+$ an absolute pairing, and $P_1,P_2$ (resp. $Q_1,Q_2$) bounded self-adjoint operators on $\H$ (resp. $\K$).
\begin{itemize}
\item[(a)] Let $\alpha_1,\alpha_2\in\IC$. If $P_i$ is dominated by $Q_i$, $i\in\{1,2\}$, then $\alpha_1P_1+\alpha_2P_2$ is dominated by $\abs{\alpha_1}Q_1+\abs{\alpha_2}Q_2$.
\item[(b)] If $\K_+$ is self-dual and $P_1$ is dominated by $Q_1$, then $Q_1$ preserves the cone $\K_+$.
\item [(c)] If $P_i$ is dominated by $Q_i$, $i\in\{1,2\}$, and $Q_1$
preserves $\K_+$, then $P_1P_2$ is dominated by $Q_1 Q_2$.
\end{itemize}
\end{lemma}

%Part (b) of the previous Lemma can be combined with Corollary
%\ref{char-pos-form-via-cone} to give the following result.
%
%\begin{corollary}[Domination implies positivity]\label{Domination-implies-positivity-eins} Let $\K_+\subset \K$ be a
%self-dual positive cone and $S\colon\H\lra\K_+$ an absolute pairing. If
%the   form $(\b,D(\b))$ in $\K$ is semibounded below and closed with
%associated self-adjoint operator $B$ such that the semigroup $e^{-t
%B}$, $t\geq 0$,  dominates a family $A_t$, $t\geq 0$,  of bounded
%self-adjoint operators in $\K$, then the form $\b$ satisfies the first Beurling-Deny criterion.
%\end{corollary}

\section{The main new  technical ingredient}\label{new}
In this section we provide  the main tool in the characterization of
domination of semigroups in terms of the associated forms. It is
given by Proposition \ref{domination_invariance}, which is an
abstract version of Lemma 3.2 in \cite{MVV}. The proof given there
relies on pointwise considerations that are not applicable in our
setting. Instead we will have to rely on the abstract properties of
absolute mappings and isotone projection cones. This makes the proof
technically demanding (and somewhat lengthy as well).

\begin{proposition}[Domination via invariance of  $C$]\label{domination_invariance}
Let $\K_+\subset\K$ be a positive cone, $S\colon \H\lra
\K_+$ an absolute mapping, and $(P_t)$ (resp. $(Q_t))$ a semigroup
on $\H$ (resp. $\K$).

Define the semigroup $(W_t)$ on $\H\oplus\K$ by
\begin{align*}
W_t(f,g)=(P_t f,Q_t g)
\end{align*}
for $t\geq 0,\,(f,g)\in\H\oplus\K$, and let
\begin{align*}
C=\{(u,v)\in\H\oplus\K\mid v-S(u)\in \K_+^\circ\}.
\end{align*}
\begin{itemize}
\item[(a)]The set $C$ is a closed, convex subset of $\H\oplus \K$.
\item[(b)]If $C$ is invariant under $(W_t)$, then $(P_t)$ is dominated by $(Q_t)$. Conversely, if $(P_t)$ is dominated by $(Q_t)$ and $(Q_t)$ leaves $\K_+$ invariant, then $C$ is invariant under $(W_t)$.
\item[(c)]Let $g\in \K_+$ and $f_1\in \H$ with $g\leq S(f_1)$. Whenever there is an $f_2\in\H$ such that $f_1,f_2$ are paired with $S(f_2)=g$, the projection $P_C$ onto $C$ satisfies
\begin{align*}
P_C(f_1,g)=\frac 1 2(f_1+f_2,S(f_1)+g).
\end{align*}
\item[(d)]If $\K_+$ is a self-dual isotone projection cone, the projection $P_C$ onto $C$ satisfies
\begin{align*}
P_C(f_1,g)=\frac 1 2(f_2, (S(f_1)\vee g+g)_+),
\end{align*}
for $f_1\in\H,g\in \K$ whenever there is an $f_2\in\H$ such that
$f_1,f_2$ are paired and $S(f_2)=(S(f_1)\wedge g+S(f_1))_+$.

\end{itemize}
\end{proposition}

\begin{remark}
\begin{itemize}
\item If $(P_t)$ is dominated by $(Q_t)$ and $\K_+$ is self-dual, then $(Q_t)$ leaves $\K_+$ invariant by Lemma \ref{domination_sums}. Thus, the domination of $(P_t)$ by $(Q_t)$ is equivalent to the invariance of $C$ under $(W_t)$ in the case of self-dual cones.
\item Of course, the existence of an element $f_2\in\H$ as in (c), (d) is
automatically guaranteed  if $S$ is actually an absolute pairing.
However, for the following corollary we need the above proposition
when $S$ is the absolute value on $\K$, which is not necessarily an absolute pairing.
\end{itemize}
\end{remark}

\begin{proof} (a) Since $S$ is positive homogeneous and satisfies the triangle
inequality (Lem\-ma \ref{symmetrization_triangle}), it is clear that
$C$ is convex. By Lemma \ref{symmetrization_triangle}, $S$ is
continuous. Thus, $C$ is closed as the preimages of $\K_+$ under the
continuous map
\begin{align*}
\phi\colon \H\oplus\K\lra \K,(u,v)\mapsto v-S(u).
\end{align*}

\smallskip

(b) First assume that $C$ is invariant under $(W_t)$. Let $f\in \H$
and $t\geq 0$. Then we have $(f,S(f))\in C$, hence
\begin{align*}
(P_t f, Q_t S(f))=W_t(f,S(f))\in C,
\end{align*}
that is, $\langle S(P_t f),g\rangle\leq \langle Q_t S(f),g\rangle$ for all $g\in \K_+$. Thus, $(P_t)$ is dominated by
$(Q_t)$.

Conversely assume that $(P_t)$ is dominated by $(Q_t)$ and that $(Q_t)$ leaves $\K_+$ invariant. If $(u,v)\in C$, then
\begin{align*}
\langle S(P_t u),g\rangle\leq \langle Q_t S(u),g\rangle = \langle S(u),Q_t g\rangle\leq \langle v,Q_t g\rangle=\langle Q_t v,g\rangle
\end{align*}
for all $g\in \K_+$. Hence,
$W_t(u,v)=(P_tu,Q_tv)\in C$.

\smallskip

(c) Let $f_1\in \H$, $g\in\K_+$ such that $g\leq S(f_1)$ and assume
there is an $f_2\in\H$ such that $f_1,f_2$ are paired with $S(f_2)=g$. Define
\begin{align*}
P(f_1,g)=\frac 1 2(f_1+f_2,S(f_1)+g).
\end{align*}
 The projection $(\hat{f_1},\hat g)$ of $(f_1,g)$ onto $C$ is characterized as the unique element in $C$ satisfying
\begin{align*}
\Re\langle (f_1,g)-(\hat{f_1},\hat g),(u,v)-(\hat{f_1},\hat
g)\rangle\leq 0
\end{align*}
for all $(u,v)\in C$. We will show that $P(f_1,g)=(\hat f_1,\hat g)$.

Since $\langle S(f_1+f_2),h\rangle\leq \langle S(f_1)+S(f_2),h\rangle=\langle S(f_1)+g,h\rangle$ for all $h\in \K_+$, we have $P(f_1,g)\in C$. For all $(u,v)\in C$ we have
\begin{align*}
&\Re\langle (f_1,g)-P(f_1,g),(u,v)-P(f_1,g)\rangle\\
={}&\frac 1 4\Re\langle (f_1-f_2,g-S(f_1)),(2u-f_1-f_2,2v-S(f_1)-g)\rangle\\
={}&\frac 1 4\Re(\langle f_1-f_2,2u\rangle-\norm{f_1}^2+\norm{f_2}^2+2\langle g-S(f_1), v\rangle-\norm{g}^2+\norm{S(f_1)}^2)\\
={}&\frac 1 2\Re(\langle f_1-f_2,u\rangle+\langle g-S(f_1),v\rangle)
\end{align*}
Since $S(f_2)=g\leq S(f_1)$, we have $S(f_1-f_2)=S(f_1)-S(f_2)$ by
Lemma \ref{abs_difference} and therefore
\begin{align*}
\abs{\langle f_1-f_2,u\rangle}\leq \langle
S(f_1-f_2),S(u)\rangle\leq \langle S(f_1)-S(f_2),v\rangle=\langle S(f_1)-g,v\rangle.
\end{align*}
This implies
\begin{align*}
\frac 1 2\Re(\langle f_1-f_2,u\rangle+\langle g-S(f_1),v\rangle)&\leq \frac 1 2(\abs{\langle f_1-f_2,u\rangle}+\langle g-S(f_1),v\rangle)\\
&\leq\frac 1 2(\langle S(f_1)-g,v\rangle+\langle g-S(f_1),v\rangle)\\
&=0.
\end{align*}
Hence, $P(f_1,g)$ is the projection of $(f_1,g)$ on $C$.

\smallskip

(d) Let $f_1\in\H,g\in\K$ and $f_2\in \H$ such that $f_1,f_2$ are
paired with $S(f_2)=(S(f_1)\wedge g+S(f_1))_+$. Define
\begin{align*}
P(f_1,g)=\frac 1 2(f_2,(S(f_1)\vee g+g)_+).
\end{align*}
As in (c), we will show $P=P_C$ via the characterization of $P_C$ given above.

Since $\K_+$ is an isotone projection cone, $S(f_1)\wedge g+S(f_1)\leq g+S(f_1)\vee g$ implies $(S(f_1)\wedge g+S(f_1))_+\leq (S(f_1)\vee g+g)_+$, hence $P(f_1,g)\in C$.

So we have to show that
\begin{align*}
\Re\langle (f_1,g)-P(f_1,g),(u,v)-P(f_1,g)\rangle\leq 0
\end{align*}
for all $(u,v)\in C$.

We will evaluate the terms
$$
I=\Re\langle (f_1,g)-P(f_1,g),(u,v)\rangle \mbox{ and }  J=\langle
(f_1,g)-P(f_1,g),-P(f_1,g)\rangle$$ separately. Using $\abs{\langle
f,\tilde f\rangle}\leq \langle S(f),S(\tilde f)\rangle$ for
$f,\tilde f\in\H$, and $S(u)\leq v$, we obtain
\begin{align*}
I&=\Re\langle f_1-\frac 1 2 f_2,u\rangle+\langle g-\frac 1 2(S(f_1)\vee g+g)_+,v\rangle\\
&\leq \langle S(f_1-\frac 1 2f_2),S(u)\rangle+\langle g-\frac 1 2(S(f_1)\vee g+g)_+,v\rangle\\
&\leq \langle S(f_1-\frac 1 2f_2)+g-\frac 1 2(S(f_1)\vee
g+g)_+,v\rangle.
\end{align*}
Lemma \ref{abs_difference} implies
\begin{align*}
S(f_1-\frac 1 2 f_2)=S(f_1)-\frac 12 S(f_2)=S(f_1)-\frac 1
2(S(f_1)\wedge g+S(f_1))_+.
\end{align*}
Thus,
\begin{align*}
I&\leq\langle S(f_1)-\frac 1 2(S(f_1)\wedge g+S(f_1))_+ +g-\frac 1 2(S(f_1)\vee g+g)_+,v\rangle\\
&\leq \langle S(f_1)+g-\frac 1 2(S(f_1)\wedge g+S(f_1))-\frac 1 2(S(f_1)\vee g+g),v\rangle\\
&=\langle S(f_1)+g-\frac 1 2 (S(f_1)+g+S(f_1)+g),v\rangle\\
&=0,
\end{align*}
where we used $h\leq h_+$ and $h\wedge \tilde h+h\vee \tilde h=h+\tilde h$ for $h,\tilde h\in \K$.

Next let us turn to $J$:
\begin{align*}
J{}={}&\langle f_1-\frac 1 2 f_2,-\frac 1 2 f_2\rangle+\langle g-\frac 1 2(S(f_1)\vee g+g)_+,-\frac 1 2(S(f_1)\vee g+g)_+\rangle\\
={}&-\frac 1 2\langle S(f_1),S(f_2)\rangle+\frac 1 4\langle S(f_2),S(f_2)\rangle-\frac 1 2\langle g,(S(f_1)\vee g+g)_+\rangle\\
{}&+\frac 1 4\langle (S(f_1)\vee g+g)_+,(S(f_1)\vee g+g)_+\rangle\\
={}&-\frac 1 2 \langle S(f_1),(S(f_1)\wedge g+S(f_1))_+\rangle+\frac  14 \norm{(S(f_1)\wedge g+S(f_1))_+}^2\\
{}&-\frac 1 2\langle g,(S(f_1)\vee g+g)_+\rangle+\frac 1
4\norm{(S(f_1)\vee g+g)_+}^2.
\end{align*}
Since positive and negative part are orthogonal to each other, we
can write
\begin{align*}
\norm{(S(f_1)\wedge g+S(f_1))_+}^2=\langle S(f_1)\wedge
g+S(f_1),(S(f_1)\wedge g+S(f_1))_+\rangle
\end{align*}
and likewise for $\norm{(S(f_1)\vee g+g)_+}^2$.

Using once again $h\wedge \tilde h+h\vee \tilde h=h+\tilde h$ for
$h,\tilde h\in \K$, it follows that
\begin{align*}
4J{}={}&\langle S(f_1)\wedge g -S(f_1),(S(f_1)\wedge g+S(f_1))_+\rangle\\
{}&+\langle S(f_1)\vee g-g,(S(f_1)\vee g+g)_+\rangle\\
={}&\langle g-S(f_1)\vee g,(S(f_1)\wedge g+S(f_1))_+\rangle+\langle S(f_1)\vee g-g,(S(f_1)\vee g+g)_+\rangle\\
={}&\langle S(f_1)\vee g-g,(S(f_1)\vee g+g)_+-(S(f_1)\wedge
g+S(f_1))_+\rangle.
\end{align*}
We analyze the factors of the inner product separately.

As for the first factor, isotonicity implies that
\begin{align*}
(S(f_1)\vee g+g)_+-(S(f_1)\wedge g+S(f_1))_+\geq
(S(f_1)+g)_+-(g+S(f_1))_+=0
\end{align*}
and
\begin{align*}
(S(f_1)\vee g+g)_+-(S(f_1)\wedge g+S(f_1))_+\geq (2g)_+-2S(f_1)\geq
2(g-S(f_1)),
\end{align*}
hence
\begin{align*}
(S(f_1)\vee g+g)_+-(S(f_1)\wedge g+S(f_1))_+\geq 2(g-S(f_1))_+.
\end{align*}
Moreover
\begin{align*}
(S(f_1)\vee g+g)-(S(f_1)\wedge g+S(f_1))&=g-S(f_1)+S(f_1)\vee g-S(f_1)\wedge g\\
&=g-S(f_1)+\abs{g-S(f_1)}\\
&=2(g-S(f_1))_+.
\end{align*}
Let $h,\tilde h\in\K$ such that $\tilde h-h\geq 0,\,\tilde
h_+-h_+\geq \tilde h-h$. Then
\begin{align*}
\tilde h_- -h_-=(\tilde h_+-h_+)-(\tilde h-h)\geq 0.
\end{align*}
On the other hand, $\tilde h\geq h$ implies by isotonicity $\tilde h_-\leq h_-$. Combining both inequalities yields $\tilde h_-=h_-$ and consequently $\tilde h_+ - h_+=\tilde h-h$.

Applied to $\tilde h=S(f_1)\vee g+g$ and $h=S(f_1)\wedge g+S(f_1)$
this means that
\begin{align*}
(S(f_1)\vee g+g)_+ -(S(f_1)\wedge g+S(f_1))_+=2(g-S(f_1))_+.
\end{align*}
For the other factor in the inner product expression for $4J$ we
have:
\begin{align*}
S(f_1)\vee g-g&=\frac 1 2(S(f_1)+g+\abs{S(f_1)-g})-g\\
&=\frac 1 2(S(f_1)-g+\abs{S(f_1)-g})\\
&=(S(f_1)-g)_+\\
&=(g-S(f_1))_-.
\end{align*}
Thus,
\begin{align*}
4J&=\langle S(f_1)\vee g-g,(S(f_1)\vee g+g)_+-(S(f_1)\wedge g+S(f_1))_+\rangle\\
&=2\langle(g-S(f_1))_-,(g-S(f_1))_+\rangle\\
&=0.
\end{align*}
Combining the results for $I$ and $J$ we finally obtain the desired
result:
\begin{align*}
\Re\langle (f_1,g)-P(f_1),(u,v)-P(f_1,g)\rangle=I+\Re J\leq
0.&\qedhere
\end{align*}
\qedhere
\end{proof}

\begin{corollary}\label{f_2_unique}
Let $\K_+\subset \K$ be a positive cone and
$S\colon\H\lra\K_+$ an absolute pairing. If $f_1\in\H$ and $g\in \K_+$
with $g\leq S(f_1)$, then the element $f_2\in\H$ such that $f_1,f_2$
are paired and $S(f_2)=g$ is unique.
\end{corollary}
\begin{proof}
This follows from Proposition \ref{domination_invariance} (c) since
the projection is unique.
\end{proof}

\begin{corollary}\label{cor_proj_pos_form}
Let $\K_+\subset \K$ be a self-dual isotone projection cone and $\b$
a form in $\K$ satisfying the first Beurling-Deny criterion. Let
\begin{align*}
\pi\colon \K_+\oplus\K\lra \K_+\oplus\K_+,\,(h,g)\mapsto \frac 1
2((h\wedge g+h)_+,(h\vee g+g)_+).
\end{align*}
Then $\pi(D(\b)\oplus D(\b))\subset D(\b)\oplus D(\b)$ and
\begin{align*}
(\b\oplus\b)(\pi(h,g),(1-\pi)(h,g))\geq 0
\end{align*}
holds for all $h\in D(\b)_+, g\in D(\b)$.
\end{corollary}
\begin{proof}
Let $(P_t)$ be the semigroup associated with $\b$. By Proposition \ref{invariance_ouhabaz}, $(P_t)$ preserves $\K_+$ and in Example \ref{ex_positive_semigroup_dominating} it was shown that $(P_t)$ is dominated by itself.

By Proposition \ref{domination_invariance},
\begin{align*}
C=\{(u,v)\in\K\oplus\K\colon \abs{u}\leq v\}
\end{align*}
is invariant under $(P_t\oplus P_t)$.

Since two positive elements of $\K$ are obviously paired (with
respect to $\abs\cdot$), we can apply Proposition
\ref{domination_invariance} to deduce that the projection onto $C$
satisfies
\begin{align*}
P(h,g)=\frac 1 2 ((h\wedge g+h)_+,(h\vee g+g)_+)=\pi(h,g)
\end{align*}
for all $g\in\K, h\in \K_+$.

One more application of Proposition \ref{invariance_ouhabaz} yields
$\pi(D(\b)\oplus D(\b))\subset D(\b)\oplus D(\b)$ and
\begin{align*}
(\b\oplus\b)(\pi(h,g),(1-\pi)(h,g))\geq 0
\end{align*}
for all $h\in D(\b)_+, g\in D(\b)$.
\end{proof}

\section{Characterizing domination of operators via forms} \label{sec_Domination_of_operators}
In this section we characterize domination of semigroups via a
domination property of forms. Roughly speaking, we do so by
combining  the framework provided in \cite{HSU} with the methods
provided in \cite{Ouh96, MVV}. More specifically, we follow the
strategy of \cite{Ouh96,MVV} and first provide below a
characterization of domination of semigroups by the invariance of a
certain convex set and then relate this invariance  to a positivity
property of the form by means of  the result of \cite{Ouh96},
Proposition \ref{invariance_ouhabaz}.

\medskip

 The  notion of generalized ideal -- to be defined next -- was  originally coined by Ouhabaz (cf.
\cite{Ouh96}) under the name ideal, but that collides with the usual
terminology in order theory, so we have adopted the usage of
\cite{MVV}.

\begin{definition}[Generalized ideal]
Let $\K_+\subset\K$ be a positive cone and $S\colon\H\lra\K_+$ an absolute pairing. A subspace $U\subset\H$ is called generalized ideal
of the subspace $V\subset\K$ if the following properties hold:
\begin{itemize}
\item[(I1)]$S(f)\in V$ for all $f\in U$,
\item[(I2)]For all $f_1\in U$ and $g\in V_+$ such that $g\leq S(f_1)$ there is an $f_2\in U$ such that $f_1,f_2$ are paired with $S(f_2)=g$.
\end{itemize}
\end{definition}

This notion is obviously closely related to that of an absolute pairing
on $U$. Notice however that contrary to the definition of a
symmetrization, we only demand the existence of $f_2\in U$ if $g\leq
S(f_1)$ here, so that $S$ does not necessarily restrict to an absolute pairing $U\lra V_+$.

\begin{remark}
Let $\K_+\subset\K$ be a positive cone and
$S\colon\H\lra\K_+$ an absolute pairing.
 Let $U$ be a generalized ideal of $V$. If $f_1\in U$, $g\in V$
such that $g\leq S(f_1)$, there is only one $f_2\in \H$ such that
$f_1,f_2$ are paired with $S(f_2)=g$ by Corollary \ref{f_2_unique}. Then
condition (I2) implies that $f_2\in U$.
\end{remark}

\begin{example}
Let $X$ be a topological space, $m$ a Borel measure on $X$ and $E$ a
Hermitian vector bundle over $X$. Then the subspace $U\subset
L^2(X,m;E)$ is a generalized ideal of the subspace $V\subset
L^2(X,m)$ if and only if
\begin{itemize}
\item $u\in U$ implies $\abs{u}\in V$,
\item $u\in U, v\in V_+, v\leq \abs{u}$ implies $v\sgn u\in U$.
\end{itemize}
\end{example}
Since $v\leq \abs{u}$ in the second condition, it is irrelevant
which value $\sgn u$ has at the zeros of $u$, and we can stick to
the usual convention $\sgn u(x)=0$ if $u(x)=0$ instead of taking
$\sgn_{\xi}$ from Example \ref{ex_sym_L2}.

\begin{definition}[Domination of forms]
Let $\K_+\subset \K$ be a positive cone, $S\colon \H\lra \K_+$ an absolute pairing, and $\a$ (resp. $\b$) a closed form on $\H$ (resp. $\K$). Then $\a$ is said to be dominated by $\b$ if $D(\a)$ is a generalized ideal of $D(\b)$ and
\begin{align*}
\Re\a(f_1,f_2)\geq \b(S(f_1),S(f_2))
\end{align*}
holds for all $f_1,f_2\in D(\a)$ that are paired.
\end{definition}

Now we can give the characterization of domination of semigroups in terms of the associated forms. For comments on the history of this theorem as well as on the proof, see the remarks below.

\begin{theorem}[Characterization of domination of semigroups via forms]
\label{thm_char_domination}
Let $\K_+\subset\K$ be a positive cone, $S\colon\H\lra\K_+$ an absolute pairing, $A$ (resp. $B$) a self-adjoint operator on $\H$ (resp. $\K)$ with lower bound $-\lambda$, and $\a$ (resp. $\b)$ the associated form. Assume that $(e^{-tB})$ leaves $\K_+$ invariant.

The following assertions are equivalent:
\begin{itemize}
\item[(i)]The semigroup $(e^{-tA})_{t\geq 0}$ is dominated by $(e^{-tB})_{t\geq 0}$.
\item[(ii)]The resolvent $((A+\alpha)^{-1})_{\alpha>\lambda}$ is dominated by $((B+\alpha)^{-1})_{\alpha>\lambda}$.
\end{itemize}
Both assertions imply
\begin{itemize}
\item[(iii)]The form $\a$ is dominated by $\b$.

\end{itemize}
If $\K_+$ is a self-dual isotone projection cone, the assertions (i), (ii) and (iii) are equivalent.
\end{theorem}
\begin{proof}
The equivalence of (i) and (ii) follows easily from the correspondence between semigroups and associated resolvents (see \cite{Ber}, Appendix A, Corollary 15).

Next we do some preparatory work for (i)$\implies$(iii) and (iii)$\implies$(i) in the isotone case.

Let $C=\{(u,v)\in\H\oplus\K\mid v-S(u)\in\K_+^\circ\}$ and $W_t=e^{-tA}\oplus e^{-tB}\colon\H\oplus \K\lra\H\oplus \K$ for $t\geq 0$. By Proposition \ref{domination_invariance} (b), (i) is equivalent to $W_t C\subset C$ for all $t\geq 0$.

The form $\tau$ associated to $(W_t)$ is given by
\begin{align*}
D(\tau)&=\{(u,v)\in\H\oplus\K\mid \lim_{t\downto 0}\frac 1 t\langle (u,v)-W_t(u,v),(u,v)\rangle<\infty\}\\
&=\{(u,v)\in\H\oplus\K\mid \lim_{t\downto 0}\frac 1 t(\langle u-e^{-tA}u,u\rangle+\langle v-e^{-tB}v,v\rangle)<\infty\}\\
&=D(\a)\oplus D(\b),\\
\tau((u,v))&=\lim_{t\downto 0}\frac 1 t(\langle
u-e^{-tA}u,u\rangle+\langle v-e^{-tB}v,v\rangle) =\a(u)+\b(v).
\end{align*}
By Proposition \ref{invariance_ouhabaz}, $C$ is invariant under $(W_t)$ if and only if $P_C D(\tau)\subset D(\tau)$ and
\begin{align*}
\Re \tau(P_C(f,g),(f,g)-P_C(f,g))\geq 0
\end{align*}
for all $(f,g)\in D(\a)\oplus D(\b)$.

(i)$\implies$(iii): By Proposition \ref{domination_invariance} (b), the projection $P_C$ satisfies
\begin{align*}
P_C(f_1,g)=\frac 1 2(f_1+f_2,S(f_1)+g)
\end{align*}
for all $f_1,f_2\in\H$, $g\in\K_+$ such that $f_1,f_2$ are paired with $S(f_2)=g$ and $g\leq S(f_1)$.

Now assume additionally that $f_1\in D(\a),\,g\in D(\b)$. If $C$ is invariant under $(W_t)$, then
\begin{align*}
P_C(f_1,g)=\frac 1 2(f_1+f_2,S(f_1)+g)\in D(\a)\oplus D(\b),
\end{align*}
hence $f_2\in D(\a)$ and $S(f_1)\in D(\b)$. Thus, $D(\a)$ is a generalized ideal of $D(\b)$.

Let $f_1,f_2\in D(\a)$ be paired. Then $S(f_1),S(f_2)\in D(\b)$ and
\begin{align*}
\Re \a(f_1,f_2)&=\lim_{t\downto 0}\frac 1 t\Re\langle f_1-e^{-tA}f_1,f_2\rangle\\
&=\lim_{t\downto 0}\frac 1 t(\langle S(f_1),S(f_2)\rangle-\Re\langle e^{-tA}f_1,f_2\rangle)\\
&\geq\lim_{t\downto 0}\frac 1 t (\langle S(f_1),S(f_2)\rangle-\langle e^{-tB}S(f_1),S(f_2)\rangle)\\
&=\lim_{t\downto 0}\frac 1 t\langle Sf_1-e^{-tB}S(f_1),S(f_2)\rangle\\
&=\b(S(f_1),S(f_2)).
\end{align*}
For the remainder of the proof we assume that $\K_+$ is a self-dual isotone projection cone. Also note that $\b$ satisfies the first Beurling-Deny criterion by Corollary \ref{char-pos-form-via-cone}. In particular, $D(\b)$ is a sublattice of $\K$ by Lemma \ref{pos_form_lattice}.

(iii)$\implies$(i): First we show that $P_C D(\tau)\subset D(\tau)$. For that purpose let $(f_1,g)\in D(\tau)$. By Proposition \ref{domination_invariance} (d), the projection $P_C(f_1,g)$ is given by
\begin{align*}
P_C(f_1,g)=\frac 1 2(f_2, (S(f_1)\vee g+g)_+),
\end{align*}
where $f_2\in \H$ such that $f_1,f_2$ are paired with $S(f_2)=(S(f_1)\wedge g+S(f_1))_+$.

By isotonicity,
\begin{align*}
S(f_2)=(S(f_1)\wedge g+S(f_1))_+\leq (2S(f_1))_+=2 S(f_1).
\end{align*}
Since $D(\a)$ is generalized ideal of $D(\b)$ and $D(\b)$ a subattice of $\K$, this inequality implies $f_2\in D(\a)$.

Furthermore, $S(f_1)\in D(\b)$ once again since $D(\a)$ is a generalized ideal in $D(\b)$ and therefore $\frac 1 2 (S(f_1)\vee g+g)_+\in D(\b)$ since $D(\b)$ is a sublattice of $\K$. Hence, $P_C(f_1,g)\in D(\tau)$.

Let $g_2=\frac 1 2(S(f_1)\vee g+g)_+$. By (iii) and Lemma \ref{abs_difference} we have
\begin{align*}
\Re \tau(P_C(f_1,g),(1-P_C)(f_1,g))&=\Re\a\left(\frac 1 2 f_2,f_1-\frac 1 2 f_2\right)+\b\left(g_2,g- g_2\right)\\
&\geq \b\left(\frac 1 2 S(f_2), S(f_1-\frac 1 2 f_2)\right)+\b(g_2,g-g_2)\\
&= \b\left(\frac 1 2 S(f_2),S(f_1)-\frac 1 2S(f_2)\right)+\b\left( g_2,g- g_2\right)\\
&=(\b\oplus\b)\left(\frac 1 2(S(f_2),g_2),(S(f_1),g)-(S(f_2),g_2)\right).
\end{align*}
Now Corollary \ref{cor_proj_pos_form} implies $(S(f_2),g_2)=\pi(S(f_1),g)$ and
\begin{align*}
(\b\oplus\b)(\pi(S(f_1),g),(1-\pi)(S(f_1),g))\geq 0.
\end{align*}
Therefore,
\begin{align*}
\Re\tau(P_C(f_1,g),(1-P_C)(f_1,g))\geq 0.
\end{align*}
In the light of our preparatory work this means that $(e^{-tA})$ is dominated by $(e^{-tB})$.
\end{proof}

\begin{remark}
\begin{itemize}
\item As already discussed in the introduction there is quite some history to a result of this form.
In particular, a first version of this theorem was given
independently by Simon (see \cite{Sim79b}, Thm. 1) for operators on
$L^2$-spaces and by Hess, Schrader, Uhlenbrock (see \cite{HSU}, Thm
2.15) in the setting of absolute pairings between abstract Hilbert
spaces (that semigroup domination implies a Kato inequality had
already been noted by Simon in \cite{Sim77}, Thm 5.1). Both articles
did not give a characterization purely in terms of forms, but the
following inequality
\begin{align*}
\Re \langle g\,\overline{\sgn f} ,Af\rangle\geq \b(\abs{f},g)
\end{align*}
for $f\in D(A),\,g\in D(\b)_+$. The characterization in terms of the
associated forms was first given by Ouhabaz \cite{Ouh96} for
semigroups on $L^2$-spaces. This was then
generalized in \cite{MVV} to vector-valued $L^2$-spaces.  In fact, \cite{MVV} is
concerned with some further ramifications of this theorem in the
case of $L^2$-spaces, allowing not necessarily densely defined,
sectorial forms and giving criteria on cores. We believe that those
carry over to our more abstract setting, however, this is not the
focus of this article.

\item If $\K_+$ is self-dual, then the assumption that $(e^{-tB})$ leaves $\K_+$ invariant is not necessary for the proof of (i)$\implies$(iii) because it is implied by (i), see Lemma \ref{domination_sums}.

\item If $\a$ is dominated by $\b$, $f_1\in D(\a)$ and $g\in D(\b)_+$ with $g\leq S(f_1)$, there is an $f_2\in D(\a)$ such that $f_1, f_2$ are paired with $S(f_2)=g$ since $D(\a)$ is a generalized ideal of $D(\b)$. The condition $g\leq S(f_1)$ cannot be dropped, as the following example shows:
Let $\E$ be the standard energy form on $\IR^n$, that is,
\begin{align*}
D(\E)=H^1(\IR^n), \E(u)=\int_{\IR^n}\abs{\nabla u}^2\,dx.
\end{align*}
Then $\E$ satisfies the first Beurling-Deny criterion, hence it is dominated by itself.  Let $f_1\in
C_c^\infty(\IR^n)$ be such that $\supp f_1\subset[-1,1]$ and
$f_1|_{[0,1]}\geq 0,\,f_1|_{[-1,0]}\leq 0$. Let $g\in
C_c^\infty(\IR^n)$, $g\geq 0$, $g|_{[-2,2]}=1$. If $f_1$ and $f_2$
are paired with $S(f_2)=g$, then $f_2(x)=g(x)\sgn f_1(x)=\sgn (x)$ for all $x\in
[-1,1]$. Hence, $f_2\notin H^1(\IR^n)$.
\item
It is interesting to note that a stronger assumption on $\K_+$ is
needed for the implication (iii)$\implies$(i) while the theorem in
\cite{HSU} works without further assumption on $\K_+$. However, it
is obvious that our proof strategy strongly relies on the fact that
$\K_+$ is an isotone projection cone and $\K$ therefore a Riesz
space. The proof in \cite{HSU} also does not carry over to our
situation.
\end{itemize}
\end{remark}

\section{Applications}\label{Applications}
In this section we present examples for our main theorem. In these
examples the Hilbert space $\H$ will be given by a Hermitian vector
bundle (see Example \ref{herm_vector_bundle}).

\medskip

For  convenience we start by  reformulating our  the theorem above
for the case of Hermitian vector bundles.

\begin{corollary}\label{cor_domination}
Let $X$ be a Lindelöf space, $m$ a Borel measure on $X$ and $E$ a Hermitian vector bundle over $X$, $A$ (resp. $B$) a lower semibounded, self-adjoint operator on $L^2(X,m;E)$ (resp. $L^2(X,m)$) and $\a$ (resp. $\b$) the associated form. Assume that $\b$ satsifies the first Beurling-Deny criterion.

The following assertions are equivalent:
\begin{itemize}
\item[(i)]The semigroup $(e^{-tA})$ is dominated by $(e^{-tB})$.
\item[(ii)]The domain $D(\a)$ is a generalized ideal of $D(\b)$ and
\begin{align*}
\Re\a(u,\tilde u)\geq \b(\abs{u},\abs{\tilde u})
\end{align*}
holds for all $u,\tilde u\in D(\a)$ satisfying $\langle u(x), \tilde u(x)\rangle_x=\abs{u(x)}_x\abs{\tilde u(x)}_x$ for almost all $x\in X$.
\end{itemize}
\end{corollary}
\begin{proof}
It suffices  to show that $u,\tilde u\in L^2(X,m;E)$ are paired if
and only if $\langle u(x),\tilde
u(x)\rangle_x=\abs{u(x)}_x\abs{\tilde u(x)}_x$ holds for almost all
$x\in X$.

So, let us assume first that $\langle u(x),\tilde
u(x)\rangle_x=\abs{u(x)}_x\abs{\tilde u(x)}_x$ holds for almost all
$x\in X$. Then,
$$ \langle u,\tilde u\rangle_{L^2(X,m;E)}=\int_X\langle u(x),\tilde u(x)\rangle_x\,dm(x)\\
=\int_X\abs{u(x)}_x\abs{\tilde u(x)}_x\,dm(x) =\langle
\abs{u},\abs{\tilde u}\rangle_{L^2(X,m)}.$$ Hence $u$ and $\tilde u$
are paired.

Conversely, assume that $u$ and $\tilde u$ are paired. Then
\begin{align*}
0=\langle \abs{u},\abs{\tilde u}\rangle_{L^2(X,m)}-\langle u,\tilde
u\rangle_{L^2(X,m;E)}=\int_X(\abs{u(x)}_x\abs{\tilde u(x)}_x-\langle
u(x),\tilde u(x)\rangle_x)\,dm(x).
\end{align*}
Moreover, the integrand is positive by Cauchy-Schwarz inequality.
Therefore it must be zero almost everywhere.
\end{proof}

We now present three classes of examples where this can be applied.

\subsection{Perturbation by potentials}
Let $(X,\mathcal{B},m)$ be a measure space and $(Q_t)$ a positivity preserving semigroup on $L^2(X,m)$ with associated form $\b$. In Example \ref{ex_positive_semigroup_dominating} it was remarked that $(Q_t)$ is dominated by itself, hence $D(\b)$ is a generalized ideal in itself and
\begin{align*}
\b(u,v)\geq\b(\abs{u},\abs{v})
\end{align*}
holds for all $u,v\in D(\b)$ satisfying $uv=\abs{u}\abs{v}$.  Now let $V\colon X\lra [0,\infty)$ be measurable
and define the form $\a$ via
\begin{align*}
D(\a)=\{u\in D(\b)\mid V^{\frac 1 2}u\in L^2(X,m)\},\,\a(u)=\b(u)+\int_X V\abs{u}^2\,dm.
\end{align*}
Then $\a$ is dominated by $\b$.

Indeed, if $u\in D(\a)\subset D(\b)$ and $v\in D(\b)$ with $v\leq \abs{u}$, then $\abs{u},v\sgn u\in D(\b)$ since $\b$ is a generalized ideal of itself. Moreover,
\begin{align*}
\int_X V\abs{v\sgn u}^2\,dm\leq \int_X V\abs{u}^2\,dm<\infty,
\end{align*}
hence $v\sgn u\in D(\a)$. Thus $D(\a)$ is a generalized ideal of $D(\b)$.

Finally, if $u,v\in D(\a)$ with $u v=\abs{u}\abs{v}$, then
\begin{align*}
 \a(u,v)=\b(u,v)+\int_V uv\,dm= \b(u,v)+\int_X V\abs{u}\abs{v}\,dm\geq b(\abs{u},\abs{v}).
\end{align*}

\subsection{Regular Schr\"odinger bundle on manifolds}
In this section we briefly discuss the setup of Schr\"odinger
bundles on manifolds.  For further background we refer to
\cite{BG17,BMS02,Gue14,HSU80} as well.

Let $(M,g,\mu)$ be a weighted Riemannian manifold, that is, $(M,g)$
is a Riemannian manifold and $\mu=e^{-\psi}\vol_g$ for some $\psi\in
C^\infty(M)$. A \emph{regular Schr\"odinger bundle} (see
\cite{BG17}) is a triple $(E,\nabla,V)$ consisting of
\begin{itemize}
\item a complex Hermitian vector bundle $E$ over $M$,
\item a metric covariant derivative $\nabla$ on $E$,
\item a potential $V\in L^1_\loc(M;\End(E))_+$.
\end{itemize}
We denote by $\Gamma(M;E)$ the space of smooth sections and by
$\Gamma_c(M;E)$ the subspace of compactly supported smooth sections.
If $E\lra M$ is Hermitian, we write $\langle\cdot|\cdot\rangle$ and
$\abs\cdot$ for the inner product and induced norm on the fibers,
respectively.

\begin{example}
If $\eta\in\Gamma(M;T^\ast M)$, then $\nabla=d+i\eta$ is a metric
covariant derivative on the trivial complex line bundle
$M\times\IC\lra M$. Thus, magnetic Schrödinger operators with
electric potential are naturally included in this setting.
\end{example}

\begin{definition}[Schrödinger form with Neumann boundary conditions]
Let
\begin{align*}
W^{1,2}(M;E)&=\{\Phi\in L^2(M;E)\mid \nabla \Phi\in L^2(M;E\otimes
T^\ast_\IC M)\},
\end{align*}
where $\nabla$ is to be understood in the distributional sense. The
space $W^{1,2}_\loc(M;E)$ is defined accordingly.

The Schrödinger form with Neumann boundary conditions $
\E^{(N)}_{\nabla,V}$ is defined by
\begin{align*}
D( \E^{(N)}_{\nabla,V})&=\{\Phi\in W^{1,2}(M;E)\mid \langle V\Phi|\Phi\rangle\in L^2(M)\},\\
 \E^{(N)}_{\nabla,V}(\Phi)&=\int_M \abs{\nabla\Phi}^2\,d\mu+\int_M\langle V\Phi|\Phi\rangle\,d\mu.
\end{align*}
\end{definition}
Just as in the scalar case one shows that $ \E^{(N)}_{\nabla,V}$ is
closed.

In the scalar case when $\nabla$ is simply the exterior derivative
on functions, we will write $\E^{(N)}_V$ for $ \E^{(N)}_{d,V}$  and
simply $\E^{(N)}$  if $V=0$. It is well-known that the forms
$\E^{(N)}_V$ is a  Dirichlet forms (see e.g. \cite{Fu}, Section
1.2). In particular, it satisfies the first Beurling-Deny criterion.
Now, the following result is proven in \cite{LSW-stabil}.

\begin{proposition}\label{domination_manifolds}
If $(E,\nabla,V)$ is a regular Schrödinger bundle and $W\in
L^1_\loc(MSW)_+$ such that $V_x\geq W_x$ in the sense of quadratic
forms for a.e. $x\in M$, then $ \E^{(N)}_{\nabla,V}$ is dominated by
$\E^{(N)}_{W}$.
%\Hmm{Added `in the sense of quadratic forms'. M}
\end{proposition}

\begin{remark}\label{kato-formal}
%\Hmm{Changed formulations slighlty, D}
 a) A distributional  version
of Kato's inequality in this setting was first proven  by Hess,
Schrader, Uhlenbrock \cite{HSU80} (for compact manifolds and
vanishing potentials) based on arguments originally due to Kato
\cite{Kat72} for magnetic Schr\"odinger operators. Their
considerations do not  include discussion of domains of the
operators or forms and, therefore,
 do
not allow one to conclude domination. Our reasoning, which relies on
the same method, can be seen as a completion of their result.

b) For open manifolds, the same domination has been proven by
G\"uneysu \cite{Gue14}, proof of Proposition 2.2) for
\emph{Dirichlet boundary conditions} (and vanishing potentials)
using methods from stochastic analysis and the semigroup
characterization of domination.

\end{remark}

\subsection{Magnetic Schr\"odinger forms on
graphs}\label{Applications-grph} In this section we will study
discrete analoga of the Laplacian respectively of  magnetic
Schr\"odinger operators in Euclidean space.

In our setup we  essentially follow the works
\cite{KL10},\cite{KL12} for graphs and Dirichlet forms over discrete
spaces and  the article \cite{MT15} for vector bundles over graphs
and magnetic Schr\"odinger operators. Further discussion can be
found in these references.

\begin{definition}[Weighted graph]
A weighted graph $(X,b,c,m)$ consists of an (at most) countable set
$X$, an edge weight $b\colon X\times X\lra [0,\infty)$, a killing
term $c\colon X\lra[0,\infty)$ and a measure $m\colon
X\lra(0,\infty)$ subject to the following conditions for all $x,y\in
X$:
\begin{itemize}
\item[(b1)]$b(x,x)=0$,
\item[(b2)]$b(x,y)=b(y,x)$,
\item[(b3)]$\sum_{z\in X}b(x,z)<\infty$.
\end{itemize}
\end{definition}

Note that our graphs are not assumed to be locally finite (i.e. to
satisfy that   $\{y\in X\mid b(x,y)>0\}$ is finite for all $y\in
X$).  We only  assume that the edge weights are  summable (at each
vertex.

We shall consider  $X$ as a discrete topological space and, hence,
write  $C_c(X)$ for  the space of functions on $X$ with finite
support.  We can (and will) consider the function  $m$  as a measure
on the power set  $\mathcal{P}(X)$ of $X$ via
\begin{align*}
m(A):=\sum_{x\in A} m(x),\;A\subset X.
\end{align*}
The corresponding $L^2$-space is denoted by $\ell^2(X,m)$.

Any graph comes with a formal Laplacian $\tilde L$ acting on
$$
\{f : X\longrightarrow \IR : \sum_{y} b(x,y) |f(y)| <\infty \mbox{
for all $x\in X$}\}$$ by
$$\tilde L f (x) = \frac{1}{m(x)} \sum_{y\in X} b(x,y) (f(x) - f(y))
+ \frac{c(x)}{m(x)} f(x)$$ for $x\in X$.

Whenever $F$ is a Hermitian vector bundle  over $X$ we write
\begin{align*}
\Gamma(X;F)&=\{u\colon X\lra\prod_{x\in X}F_x\mid u(x)\in F_x\},\\
\Gamma_c(X;F)&=\{u\in\Gamma(X;F)\mid \supp u\text{ finite}\},\\
\ell^2(X,m;F)&=\{u\in\Gamma(X;F)\mid\sum_{x\in X}\langle
u(x),u(x)\rangle_x m(x)<\infty\}
\end{align*}
for the space of all sections, the space of all sections with
compact support and the space of all $L^2$-sections respectively.
The latter becomes a Hilbert space when  equipped with the inner
product
\begin{align*}
\langle
\cdot,\cdot\rangle_{\ell^2(X,m;F)}\colon\ell^2(X,m;F)\times\ell^2(X,m;F)\lra\IC,\,(u,v)\mapsto\sum_{x\in
X}\langle u(x),v(x)\rangle_x m(x).
\end{align*}
A bundle endomorphism $W$ of a Hermitian vector bundle $F$ is a
family of linear maps $(W(x)\colon F_x\lra F_x)_{x\in X}$.

\smallskip

Consider now a weighted graph $(X,b,c,m)$, a Hermitian vector bundle
$F$  over $X$ and $W$ a bundle endomorphism of $F$ that is pointwise
positive, that is, $\langle W(x)v,v\rangle_x\geq 0$ for all $x\in
X,\,v\in F_x$. Let moreover a family of unitary maps (called
connection maps) $\Phi_{x,y}\colon F_y\lra F_x$ for all $x,y\in X$
be given  such that $\Phi_{x,y}=\Phi_{y,x}^{-1}$.

\smallskip

Now we can define the basic object of our interest, the magnetic
Schr\"odinger form (with Dirichlet and Neumann boundary conditions).

\begin{definition}[Magnetic form with Neumann boundary conditions]\label{def_magnetic_graph_Neumann}
For $u\in \Gamma(X;F)$ let
\begin{align*}
\tilde Q_{\Phi,b,W}(u)=\frac 1
2\sum_{x,y}b(x,y)\abs{u(x)-\Phi_{x,y}u(y)}_x^2+\sum_{x}\langle
W(x)u(x),u(x)\rangle_x\in[0,\infty].
\end{align*}
The magnetic Schr\"odinger form with Neumann boundary conditions is
defined via
\begin{align*}
D(Q^{(N)}_{\Phi,b,W})&=\{u\in\ell^2(X,m;F)\mid\tilde Q_{\Phi,b,W}(u)<\infty\},\\
Q^{(N)}_{\Phi,b,W}(u)&=\tilde Q_{\Phi,b,W}(u).
\end{align*}
\end{definition}

To abridge notation, we will write $\norm{\cdot}_{\Phi,b,W}$ for the
form norm of $Q^{(N)}_{\Phi,b,W}$. By the same arguments as in the
Dirichlet form case, the form $Q^{(N)}_{\Phi,b,W}$ is closed (see
\cite{KL12}, Lemma 2.3). The form   $Q^{(N)}_{b,c}$ is of particular
interest as it is  a Dirichlet form. Given these preparations we can
now state the following result from \cite{LSW-stabil}.

\begin{proposition}\label{domination_magnetic_forms}
Assume that $\langle W(x)u(x),u(x)\rangle_x\geq c(x)\abs{u(x)}^2$
for all $x\in X,\,u(x)\in F_x$. Then $Q^{(N)}_{\Phi,b,W}$ is
dominated by $Q^{(N)}_{b,c}$.
\end{proposition}

\begin{corollary}
The form $Q_{b,c}^{(N)}$ is dominated by $Q^{(N)}_{b,0}$.
\end{corollary}
\begin{proof}
This follows from Proposition \ref{domination_magnetic_forms} by
taking $F_x=\IC$, $W(x)=c(x)$ and $\Phi_{x,y}=1$ for all $x,y\in X$.
\end{proof}

\end{document}